\documentclass[10pt]{amsart}
\usepackage{amssymb}
\usepackage{amsfonts}
\usepackage{amsmath}
\usepackage[centertags]{amsmath}
\usepackage{newlfont}
\usepackage{url}

\newtheorem{theorem}[subsection]{Theorem}
\theoremstyle{plain}

\newtheorem{corollary}[subsection]{Corollary}

\numberwithin{equation}{subsection}

\theoremstyle{plain}

\theoremstyle{definition}

\theoremstyle{remark}

\theoremstyle{plain}

\input{tcilatex}

\begin{document}
\title{\textbf{On the Derivatives of Bivariate Fibonacci Polynomials}}
\author{Tuba \c{C}akmak and Erdal Karaduman}
\address{Atat\"{u}rk University, Faculty of Science, Department of
Mathematics, 25240 Erzurum, Turkey }
\email{cakmaktuba@yahoo.com, eduman@atauni.edu.tr}

\begin{abstract}
{In this study, the new algebraic properties related to bivariate Fibonacci
polynomials has been given. We present the partial derivatives of these
polynomials in the form of convolution of bivariate Fibonacci polynomials.
Also, we define a new recurrence relation for the }${r}${-th partial
derivative sequence of bivariate Fibonacci polynomials.}
\end{abstract}

\keywords{ ${k}${-Fibonacci sequences, Bivariate Fibonacci polynomials,
Partial derivatives of bivariate Fibonacci polynomials}}
\maketitle

\section{Introduction}

In modern science, there is a huge interest in the theory and application of
the Golden Section and Fibonacci numbers in [1-17]. The Fibonacci numbers $%
F_{n}$\ are the terms of the sequence $0,1,1,2,5,...$ where $%
F_{n}=F_{n-1}+F_{n-2},$ $n\geq 2,$ with the initial values $F_{0}=0$ and $%
F_{1}=1$. Falcon and Plaza \cite{4} introduced a general Fibonacci sequence
that generalizes the classical Fibonacci sequence. These general $k$%
-Fibonacci numbers $F_{k,n}$ are defined by $F_{k,n}=kF_{k,n-1}+F_{k,n-2},$ $%
n\geq 2,$ with the initial conditions $F_{k,0}=0$ and $F_{k,1}=1$. If $k$ is
a real variable then $F_{k,n}$ will be equal to $F_{x,n}$ and they
correspond to Fibonacci polynomials.

\bigskip

Fibonacci polynomials were studied in 1883 by the Belgian mathematician
Eugene Charles and German mathematician E. Jacobsthal. The polynomials $%
F_{n}(x)$ were defined by the recurrence relation\ 
\begin{equation*}
F_{n}(x)=xF_{n-1}(x)+F_{n-2}(x),\text{ \ \ }n\geq 2
\end{equation*}%
where $F_{0}(x)=0$ and $F_{1}(x)=1$. The Fibonacci polynomials and their
relationship to diagonals of Pascal's triangle were generalized by Hoggat
and B\i cknell \cite{5}. Some relationships between Zeckendorf theorem and
Fibonacci polynomials \cite{6} were examined. Their divisibility properties
had been found by Gerard Jacob \cite{10} and Webb \cite{17}, Hoggat and B\i
cknell \cite{7} found the roots of Fibonacci polynomials of degree n. In 
\cite{11}, $h(x)$-Fibonacci polynomials were defined that generalize both
Catalan's and Bryd's Fibonacci polyomials. In \cite{4}, Falcon and Plaza
investigated the derivatives of these polynomials and they had given many
relations for the derivatives of Fibonacci polynomials.

\bigskip

Afterwards, some new generalizations were identified about Fibonacci
polynomials which were given by Catalan. One of them is bivariate Fibonacci
polynomials defined as%
\begin{equation*}
F_{n}(x,y)=xF_{n-1}(x,y)+yF_{n-2}(x,y),\text{ \ \ }n\geq 2.
\end{equation*}%
Many properties for selected values of the variables and some recurrence
relations of bivariate Fibonacci and Lucas polynomials were obtained \cite{1}%
. In \cite{9}, the properties of bivariate Fibonacci polynomials order $k$
had been investigated in terms of the generating functions. M.N.S.Swamy \cite%
{12} derived some new properties concerning the derivatives of bivariate
Fibonacci and Lucas polynomials. In \cite{2, 3} the works of Filipponi and
Horadam revealed the first and second order derivative sequences of
Fibonacci and Lucas polynomials and these results had been extended to the $%
k $-th derivative case as conjectured in \cite{3} and then had been
confirmed in \cite{16}. Filipponi and Horadam \cite{3} considered the
partial derivative sequences of bivariate second order recurrence
polynomials. In \cite{15}, Yu and Liang extended some of the results and
derived some identities involving the partial derivative sequences of
bivariate Fibonacci and Lucas polynomials. One of the generalization of
Fibonacci polynomials is given by Tu\u{g}lu, Ko\c{c}er and Stakhov \cite{13}%
. Also in \cite{14}, Claudio de Jesus Pita Ruiz Velauco define bivariate $s$%
-Fibopolynomials and give some of the derivative identities.

\bigskip

This paper is based on the definition of Falcon and Plaza \cite{4} and Tu%
\u{g}lu, Ko\c{c}er, Stakhov \cite{13}. We study on the derivatives of
bivariate Fibonacci polynomials in the form of convolution of these
polynomials. In this sense the present paper is organised as follows. In
Section 2, a brief summary of the previous results obtained by Falcon and
Plaza in \cite{4} is given. In Section 3, some basic facts are given about
bivariate Fibonacci Polynomials. Section 4 presents some new relations
related with the derivatives of bivariate Fibonacci Polynomials and gives a
new recurrence relation for the $r$-th partial derivative sequence.

\bigskip 

\section{The Fibonacci Polynomials}

In this section, we will give some basic facts related to Fibonacci
polynomials. These results and more can be found in \cite{4}.

The $k$-Fibonacci sequence, namely $\{F_{k,n}\}_{n\in 
\mathbb{N}
}$ have been defined recurrently by $F_{k,n+1}=kF_{k,n}+F_{k,n-1}$ for $%
n\geq 1$ and any positive real number $k$, with initial conditions $%
F_{k,0}=0,F_{k,1}=1$. If $k$ is a real variable $x$ then $F_{k,n}=F_{x,n}$
and they correspond to the Fibonacci polynomials. The Fibonacci polynomials
are defined as follow 
\begin{equation}
F_{n+1}\left( x\right) =\left\{ 
\begin{array}{c}
1,\text{ \ \ \ \ \ \ \ \ \ \ \ \ \ \ \ \ \ \ \ \ \ \ \ \ \ \ \ \ }n=0 \\ 
x,\text{ \ \ \ \ \ \ \ \ \ \ \ \ \ \ \ \ \ \ \ \ \ \ \ \ \ \ \ \ }n=1 \\ 
xF_{n}\left( x\right) +F_{n-1}\left( x\right) ,\text{ \ }n\geq 2%
\end{array}%
\right. .  \tag{2.1}  \label{2.1}
\end{equation}%
According to this definition the sequence of the Fibonacci polynomials is%
\begin{equation*}
\{F_{n}(x)\}=\{1,x,x^{2}+1,x^{3}+2x,x^{4}+3x^{2}+1,\ldots \}.
\end{equation*}%
Note that the $k$-Fibonacci polynomials are the natural extension of the $k$%
-Fibonacci numbers. The general term of the Fibonacci polynomials is 
\begin{equation}
F_{n+1}(x)=\dsum\limits_{i=0}^{\lfloor \frac{n}{2}\rfloor }\binom{n-i}{i}%
x^{n-2i},\text{ \ \ }n\geq 0.  \tag{2.2}  \label{2.2}
\end{equation}%
On the other hand, by deriving the elements of the sequence of Fibonacci
polynomials the following derivative sequence is obtained: 
\begin{equation*}
\{F_{n}^{^{\prime }}(x)\}=\{0,1,2x,3x^{2}+2,4x^{3}+6x,\ldots \}.
\end{equation*}%
The general term of the derivative sequence of Fibonacci polynomials is
given as 
\begin{equation}
F_{n+1}^{^{\prime }}(x)=\dsum\limits_{i=0}^{\lfloor \frac{n-1}{2}\rfloor }%
\binom{n-i}{i}\left( n-2i\right) x^{n-1-2i},\text{ \ \ }n\geq 1  \tag{2.3}
\label{2.3}
\end{equation}%
by deriving equation \ref{2.2} where $F_{1}^{^{\prime }}(x)=0$.

\bigskip 

\section{\protect\bigskip Bivariate Fibonacci Polynomials}

One of the generalization of Fibonacci type polynomials and also Fibonacci
numbers is bivariate Fibonacci polynomials. In \cite{12}, the generalized
bivariate Fibonacci polynomials are defined as 
\begin{equation}
H_{n}(x,y)=xH_{n-1}(x,y)+yH_{n-2}(x,y),\text{ \ \ }n\geq 2.  \tag{3.1}
\label{3.1}
\end{equation}%
with initial conditions $H_{0}(x,y)=a_{0},$ $H_{1}(x,y)=a_{1}$ and it is
assumed $y\neq 0$ as well as $x^{2}+4y\neq 0$. Catalani \cite{12}, set $%
a_{0}=0$, $a_{1}=1$ and he obtained the bivariate Fibonacci polynomials $%
F_{n}(x,y)$; then with $a_{0}=2$, $a_{1}=x$ he obtained the bivariate Lucas
polynomials $L_{n}(x,y)$, where,%
\begin{equation}
F_{n}(x,y)=xF_{n-1}(x,y)+yF_{n-2}(x,y),\text{ \ \ }n\geq 2  \tag{3.2}
\label{3.2}
\end{equation}%
and%
\begin{equation}
L_{n}(x,y)=x_{n-1}L(x,y)+yL_{n-2}(x,y),\text{ \ \ }n\geq 2.  \tag{3.3}
\label{3.3}
\end{equation}

So, the first bivariate Fibonacci polynomials are \newline
\begin{equation*}
\{F_{n}(x,y)\}=\{0,1,x,x^{2}+y,x^{3}+2xy,x^{4}+3x^{2}y+y^{2},\ldots \}.
\end{equation*}%
In \cite{8}, the general term of bivariate Fibonacci polynomials has been
given by

\begin{equation}
F_{n+1}(x,y)=\dsum\limits_{i=0}^{\lfloor \frac{n}{2}\rfloor }\binom{n-i}{i}%
x^{n-2i}y^{i},\text{ \ \ }n\geq 0.  \tag{3.4}  \label{3.4}
\end{equation}%
Some of relations between bivariate Fibonacci and Lucas polynomials as
follows:

\bigskip

\begin{itemize}
\item The following equation appear in \cite[eq. (2.8)]{12} for all $n\in 
\mathbb{Z}
$,%
\begin{equation}
L_{n}(x,y)=F_{n+1}(x,y)+yF_{n-1}(x,y),  \tag{3.5}  \label{3.5}
\end{equation}

\item The following equation appear in \cite[eq. (2.11)]{12} for all $n\in 
\mathbb{Z}
$, 
\begin{equation}
\left( x^{2}+4y\right) F_{n}(x,y)=L_{n+1}(x,y)+yL_{n-1}(x,y)  \tag{3.6}
\label{3.6}
\end{equation}

\item The following equation can be seen in \cite[eq. (3.10)]{12}%
\begin{equation}
\frac{\partial L_{n}(x,y)}{\partial x}=nF_{n}(x,y).  \tag{3.7}  \label{3.7}
\end{equation}
\end{itemize}

\section{Expression of the Derivative of Bivariate Fibonacci Polynomials}

In this section, we establish many formulas and relations for the
derivatives of the bivariate Fibonacci polynomials. Their derivatives are
given as convolution of bivariate Fibonacci polynomials. This fact allows us
to present a family of integer sequences in a new and direct way. Also, we
give a new recurrence relation for the $r$-th partial derivative sequence.

If the equation \ref{3.4} differantiate with respect to $x$ and $y$ the
folllowing equalities are obtained

\bigskip\ 
\begin{equation}
\frac{\partial F_{n+1}(x,y)}{\partial x}=\dsum\limits_{i=0}^{\lfloor \frac{%
n-1}{2}\rfloor }\binom{n-i}{i}\left( n-2i\right) x^{n-1-2i}y^{i},\text{ \ \ }%
n\geq 1  \tag{4.1}  \label{4.1}
\end{equation}

\begin{equation}
\frac{\partial F_{n+1}(x,y)}{\partial y}=\dsum\limits_{i=0}^{\lfloor \frac{%
n-1}{2}\rfloor }\binom{n-i}{i}\left( i\right) x^{n-2i}y^{i-1},\text{ \ \ }%
n\geq 1  \tag{4.2}  \label{4.2}
\end{equation}%
where $\frac{\partial F_{1}(x,y)}{\partial x}=\frac{\partial F_{1}(x,y)}{%
\partial y}=0$ and by \cite[eq. (3.12)]{12} 
\begin{equation}
\frac{\partial F_{n}(x,y)}{\partial x}=\frac{\partial F_{n+1}(x,y)}{\partial
y}  \tag{4.3}  \label{4.3}
\end{equation}%
is written.

\bigskip

The following theorem is a special form of \cite[Theorem (2.a)]{15}.

\bigskip

\begin{theorem}
If $\frac{\partial F_{1}(x,y)}{\partial x}=0$, for $n>1,$ then 
\begin{equation*}
\frac{\partial F_{n}(x,y)}{\partial x}=\dsum%
\limits_{i=1}^{n-1}F_{i}(x,y)F_{n-i}(x,y).
\end{equation*}
\end{theorem}

\begin{proof}
We prove this conclusion by induction. For $n=2$ it is trivial, since%
\begin{equation*}
\dsum\limits_{i=1}^{2-1}F_{i}(x,y)F_{2-i}(x,y)=F_{1}(x,y)F_{1}(x,y)=1=\frac{%
\partial }{\partial x}\left( x\right) =\frac{\partial F_{2}(x,y)}{\partial x}%
.
\end{equation*}%
Let us suppose that the formula is true for $k\leq n$. Then 
\begin{eqnarray*}
\frac{\partial F_{n-1}(x,y)}{\partial x} &=&\dsum%
\limits_{i=1}^{n-2}F_{i}(x,y)F_{n-1-i}(x,y) \\
\frac{\partial F_{n}(x,y)}{\partial x} &=&\dsum%
\limits_{i=1}^{n-1}F_{i}(x,y)F_{n-i}(x,y).
\end{eqnarray*}%
By deriving the equation \ref{3.2} according to variable $x$ and using
previous expression we get 
\begin{eqnarray*}
\frac{\partial F_{n+1}(x,y)}{\partial x} &=&F_{n}(x,y)+x\frac{\partial
F_{n}(x,y)}{\partial x}+y\frac{\partial F_{n-1}(x,y)}{\partial x} \\
&=&F_{n}(x,y)+x\left( \dsum\limits_{i=1}^{n-1}F_{i}(x,y)F_{n-i}(x,y)\right)
+y\left( \dsum\limits_{i=1}^{n-2}F_{i}(x,y)F_{n-1-i}(x,y)\right)  \\
&=&F_{n}(x,y)+xF_{1}(x,y)F_{n-1}(x,y)+\dsum\limits_{i=1}^{n-1}F_{i}(x,y)
\left[ xF_{n-i}(x,y)+yF_{n-1-i}(x,y)\right]  \\
&=&F_{n}(x,y)F_{1}(x,y)+F_{n-1}(x,y)F_{2}(x,y)+\dsum%
\limits_{i=1}^{n-2}F_{i}(x,y)F_{n-i+1}(x,y) \\
&=&\dsum\limits_{i=1}^{n}F_{i}(x,y)F_{n-i+1}(x,y).
\end{eqnarray*}
\end{proof}

\bigskip

\begin{corollary}
For $n>1$%
\begin{equation*}
\frac{\partial F_{n}(x,y)}{\partial y}=\dsum%
\limits_{i=1}^{n-2}F_{i}(x,y)F_{n-1-i}(x,y)
\end{equation*}%
where $\frac{\partial F_{1}(x,y)}{\partial x}=0$.
\end{corollary}

\bigskip

\begin{theorem}
\begin{equation*}
\frac{\partial F_{n+1}(x,y)}{\partial x}=\dsum\limits_{i=0}^{\lfloor \frac{%
n-1}{2}\rfloor }\left( -1\right) ^{i}\left( n-2i\right) F_{n-2i}(x,y)y^{i},%
\text{ \ \ }n\geq 1.
\end{equation*}
\end{theorem}

\begin{proof}
We proof this by induction. It is clear that the claim is true for $n=1$.
Let us suppose that the claim is true for $n$. If n is an even integer by
induction hypothesis we get 
\begin{equation*}
\frac{\partial F_{2p+1}(x,y)}{\partial x}=\dsum\limits_{i=0}^{p-1}\left(
-1\right) ^{i}\left( 2p-2i\right) F_{2p-2i}(x,y)y^{i}
\end{equation*}%
and 
\begin{eqnarray*}
\frac{\partial F_{2p}(x,y)}{\partial x} &=&\dsum\limits_{i=0}^{p-1}\left(
-1\right) ^{i}\left( 2p-1-2i\right) F_{2p-1-2i}(x,y)y^{i} \\
&=&\dsum\limits_{i=0}^{p-1}\left( -1\right) ^{i}\left( 2p-2i\right)
F_{2p-1-2i}(x,y)y^{i}+\dsum\limits_{i=0}^{p-1}\left( -1\right) ^{1+i}\left(
2p-1-2i\right) y^{i}.
\end{eqnarray*}%
Now, we will show the equation is true for $n+1$ namely $2p+1$. By using \ref%
{3.2} we get 
\begin{equation*}
F_{2p+2}(x,y)=xF_{2p+1}(x,y)+yF_{2p}(x,y).
\end{equation*}%
Then, 
\begin{eqnarray*}
\frac{\partial F_{2p+2}(x,y)}{\partial x} &=&F_{2p+1}(x,y)+x\frac{\partial
F_{2p+1}(x,y)}{\partial x}+y\frac{\partial F_{2p}(x,y)}{\partial x} \\
&=&F_{2p+1}(x,y)+x\dsum\limits_{i=0}^{p-1}\left( -1\right) ^{i}\left(
2p-2i\right) y^{i}\left[ xF_{2p-2i}(x,y)+yF_{2p-1-2i}(x,y)\right]  \\
&&+y\dsum\limits_{i=0}^{p-1}\left( -1\right) ^{1+i}\left( 2p-1-2i\right)
y^{i} \\
&=&\left( 2p+1\right) F_{2p+1}(x,y)-p\left( 2p-1\right)
F_{2p-1}(x,y)+...+\left( -1\right) ^{p}F_{1}(x,y)y^{p} \\
&=&\dsum\limits_{i=0}^{p}\left( -1\right) ^{i}\left( 2p+1-2i\right)
F_{2p+1-2i}(x,y)y^{i} \\
&=&\dsum\limits_{i=0}^{\lfloor \frac{n-1}{2}\rfloor }\left( -1\right)
^{i}\left( n-2i\right) F_{n-2i}(x,y)y^{i}=\frac{\partial F_{n+1}(x,y)}{%
\partial x}.
\end{eqnarray*}%
The similar proof can be given for the case $n$ is an odd integer. Thus the
result follows for all natural numbers. 
\end{proof}

\bigskip

\begin{corollary}
For \ $n\geq 3$%
\begin{equation*}
\frac{\partial F_{n}(x,y)}{\partial y}=\dsum\limits_{i=0}^{\lfloor \frac{n-3%
}{2}\rfloor }\left( n-2-2i\right) F_{n-2-2i}(x,y)y^{i}.
\end{equation*}
\end{corollary}

\bigskip

\begin{theorem}
\label{4.5}%
\begin{equation*}
F_{n}(x,y)=\frac{1}{n}\left[ \frac{\partial }{\partial x}F_{n+1}(x,y)+y\frac{%
\partial }{\partial x}F_{n-1}(x,y)\right] .
\end{equation*}
\end{theorem}

\begin{proof}
By using equation \ref{3.2} it can be written 
\begin{equation*}
F_{n+1}(x,y)=\dsum\limits_{i=0}^{\lfloor \frac{n}{2}\rfloor }\binom{n-i}{i}%
\left( x\right) ^{n-2i}y^{i},\text{ \ \ }n\geq 1
\end{equation*}%
and%
\begin{equation*}
F_{n-1}(x,y)=\dsum\limits_{i=0}^{\lfloor \frac{n-2}{2}\rfloor }\binom{n-2-i}{%
i}\left( x\right) ^{n-2-2i}y^{i},\text{ \ \ }n\geq 3.
\end{equation*}%
Then%
\begin{eqnarray*}
F_{n+1}(x,y)+yF_{n-1}(x,y) &=&\dsum\limits_{i=0}^{\lfloor \frac{n}{2}\rfloor
}\binom{n-i}{i}\left( x\right) ^{n-2i}y^{i}+y\dsum\limits_{i=0}^{\lfloor 
\frac{n-2}{2}\rfloor }\binom{n-2-i}{i}\left( x\right) ^{n-2-2i}y^{i} \\
&=&x^{n}+\dsum\limits_{i=1}^{\lfloor \frac{n}{2}\rfloor }\binom{n-i}{i}%
\left( x\right) ^{n-2i}y^{i}+\dsum\limits_{i=0}^{\lfloor \frac{n-2}{2}%
\rfloor }\binom{n-2-i}{i}\left( x\right) ^{n-2-2i}y^{i+1} \\
&=&x^{n}+\dsum\limits_{i=1}^{\lfloor \frac{n}{2}\rfloor }\binom{n-i}{i}%
\left( x\right) ^{n-2i}y^{i}+\dsum\limits_{i=1}^{\lfloor \frac{n-2}{2}%
\rfloor }\binom{n-1-i}{i-1}\left( x\right) ^{n-2-2i}y^{i} \\
&=&x^{n}+\dsum\limits_{i=1}^{\lfloor \frac{n}{2}\rfloor }\left[ \binom{n-i}{i%
}+\binom{n-1-i}{i-1}\right] \left( x\right) ^{n-2i}y^{i} \\
&=&x^{n}+n\dsum\limits_{i=1}^{\lfloor \frac{n}{2}\rfloor }\binom{n-1-i}{i-1}%
\frac{1}{i}\left( x\right) ^{n-2i}y^{i}.
\end{eqnarray*}%
Now by deriving the last form, we get 
\begin{equation*}
\frac{\partial }{\partial x}F_{n+1}(x,y)+y\frac{\partial }{\partial x}%
F_{n-1}(x,y)=nx^{n-1}+n\dsum\limits_{i=1}^{\lfloor \frac{n}{2}\rfloor }%
\binom{n-1-i}{i-1}\frac{n-2i}{i}\left( x\right) ^{n-1-2i}y^{i}.
\end{equation*}%
Then, 
\begin{eqnarray*}
\frac{1}{n}\left[ \frac{\partial }{\partial x}F_{n+1}(x,y)+y\frac{\partial }{%
\partial x}F_{n-1}(x,y)\right]  &=&x^{n-1}+\dsum\limits_{i=1}^{\lfloor \frac{%
n}{2}\rfloor }\binom{n-1-i}{i-1}\frac{n-2i}{i}\left( x\right) ^{n-1-2i}y^{i}
\\
&=&\dsum\limits_{i=0}^{\lfloor \frac{n-1}{2}\rfloor }\binom{n-1-i}{i}\left(
x\right) ^{n-1-2i}y^{i}=F_{n}(x,y)
\end{eqnarray*}

So we are done. 
\end{proof}

\bigskip

\begin{theorem}
\begin{equation*}
\frac{\partial F_{n}(x,y)}{\partial x}=\frac{\left( n+1\right)
F_{n+1}(x,y)+y\left( n-1\right) F_{n-1}(x,y)-2xF_{n}(x,y)}{x^{2}+4y},\text{
\ \ }n\geq 1.
\end{equation*}
\end{theorem}

\begin{proof}
By deriving equation \ref{3.6} according to variable $x$ we get%
\begin{equation*}
2xF_{n}(x,y)+\left( x^{2}+4y\right) \frac{\partial F_{n}(x,y)}{\partial x}=%
\frac{\partial L_{n+1}(x,y)}{\partial x}+y\frac{\partial L_{n-1}(x,y)}{%
\partial x}.
\end{equation*}%
At this point, by using equation \ref{3.7} the conclusion can be seen. 
\end{proof}

\bigskip

\begin{corollary}
For $n>1$%
\begin{equation*}
\frac{\partial F_{n}(x,y)}{\partial x}=\frac{\left( n\right)
F_{n}(x,y)+y\left( n-2\right) F_{n-2}(x,y)-2xF_{n-1}(x,y)}{x^{2}+4y}.
\end{equation*}
\end{corollary}

\bigskip

\begin{theorem}
\begin{equation*}
\frac{\partial ^{r}F_{n+1}(x,y)}{\partial x^{r}}=\left\{ 
\begin{array}{l}
\text{For }n<r,\text{ \ \ \ \ \ \ \ \ \ \ \ \ \ \ \ \ \ \ \ \ \ \ \ \ \ \ \
\ \ \ \ \ \ \ \ \ \ }0\text{ \ \ \ \ \ \ \ \ \ \ \ \ \ \ \ \ \ \ \ \ \ \ \ \
\ \ \ \ \ \ \ \ \ \ \ \ \ \ \ \ \ \ \ \ \ \ \ \ \ \ \ \ \ \ \ \ \ \ \ \ \ \
\ \ \ \ } \\ 
\text{For }n=r,\text{ \ \ \ \ \ \ \ \ \ \ \ \ \ \ \ \ \ \ \ \ \ \ \ \ \ \ \
\ \ \ \ \ \ \ \ \ \ }r!\text{ \ \ \ \ \ \ \ \ \ \ \ \ \ \ \ \ \ \ \ \ \ \ \
\ \ \ \ \ \ \ \ \ \ \ \ \ \ \ \ \ \ \ \ \ \ \ \ \ \ \ \ \ \ \ \ \ \ \ \ \ \
\ } \\ 
\text{For }n>r,\text{ \ \ \ \ \ \ \ \ }\frac{1}{n-r}\left[ nx\frac{\partial
^{r}F_{n}(x,y)}{\partial x^{r}}+y\left( n+r\right) \frac{\partial
^{r}F_{n-1}(x,y)}{\partial x^{r}}\right]%
\end{array}%
\right. .
\end{equation*}
\end{theorem}

\begin{proof}
We proof this by induction. For $r=1$, three subcases occur; 

\begin{itemize}
\item If $n<1$, then $\frac{\partial F_{1}(x,y)}{\partial x}=0$,

\item If $n=1$, then $\frac{\partial F_{2}(x,y)}{\partial x}=\left( x\right)
^{^{\prime }}=1!$,

\item If $n>1$, then by deriving the equation \ref{3.2} according to
variable $x$ we get 
\begin{equation*}
\frac{\partial F_{n+1}(x,y)}{\partial x}=F_{n}(x,y)+x\frac{\partial
F_{n}(x,y)}{\partial x}+y\frac{\partial F_{n-1}(x,y)}{\partial x}
\end{equation*}%
and by Theorem \ref{4.5} the following conclusion is obtained: 
\begin{equation*}
\frac{\partial F_{n+1}(x,y)}{\partial x}=\frac{1}{n-1}\left[ nx\frac{%
\partial F_{n}(x,y)}{\partial x}+y\left( n+1\right) \frac{\partial
F_{n-1}(x,y)}{\partial x}\right] 
\end{equation*}%
So, the claim holds for $r=1$. Let us suppose that the equation is true for $%
r$-th partial derivative and we will induct on $r$. 

\item If $n<r+1$, from induction hypothesis we can write $\frac{\partial
^{r+1}F_{n+1}(x,y)}{\partial x^{r+1}}=0$,

\item If $n=r+1$, by deriving the equation \ref{3.2} ) $\left( r+1\right) $
times according to variable $x$ we get 
\begin{eqnarray}
\frac{\partial ^{r+1}F_{n+1}(x,y)}{\partial x^{r+1}} &=&\left( r+1\right) 
\frac{\partial ^{r}F_{n}(x,y)}{\partial x^{r}}+x\frac{\partial
^{r+1}F_{n}(x,y)}{\partial x^{r+1}}+y\frac{\partial ^{r+1}F_{n-1}(x,y)}{%
\partial x^{r+1}}  \TCItag{4.4}  \label{4.4} \\
&=&\left( r+1\right) r!=\left( r+1\right) !  \notag
\end{eqnarray}

\item If $n>r+1$, by considering the induction hypothesis%
\begin{equation*}
\frac{\partial ^{r}F_{n+1}(x,y)}{\partial x^{r}}=\frac{1}{n-r}\left[ nx\frac{%
\partial ^{r}F_{n}(x,y)}{\partial x^{r}}+y\left( n+r\right) \frac{\partial
^{r}F_{n-1}(x,y)}{\partial x^{r}}\right] .
\end{equation*}%
If this equation is derived according to variable $x$ we get 
\begin{equation}
\frac{\partial ^{r+1}F_{n+1}(x,y)}{\partial x^{r+1}}=\frac{1}{n-r}\left[ n%
\frac{\partial ^{r}F_{n}(x,y)}{\partial x^{r}}+nx\frac{\partial
^{r+1}F_{n}(x,y)}{\partial x^{r+1}}+y\left( n+r\right) \frac{\partial
^{r+1}F_{n-1}(x,y)}{\partial x^{r+1}}\right] .  \tag{4.5}  \label{4.6}
\end{equation}%
If the equations \ref{4.4} and \ref{4.6} are considered together the desired
recurrence relation can be obtained: 
\begin{equation*}
\frac{\partial ^{r+1}F_{n+1}(x,y)}{\partial x^{r+1}}=\frac{1}{n-r-1}\left[ nx%
\frac{\partial ^{r+1}F_{n}(x,y)}{\partial x^{r+1}}+y\left( n+r+1\right) 
\frac{\partial ^{r+1}F_{n-1}(x,y)}{\partial x^{r+1}}\right] 
\end{equation*}
\end{itemize}
\end{proof}

\bigskip

\begin{corollary}
For $n>1$%
\begin{equation*}
\frac{\partial ^{r}F_{n+1}(x,y)}{\partial x^{r}}=\left\{ 
\begin{array}{l}
\text{For }n<r,\text{ \ \ \ \ \ \ \ \ \ \ \ \ \ \ \ \ \ \ \ \ \ \ \ \ \ \ \
\ \ \ \ \ \ \ \ \ \ \ \ \ \ \ \ \ \ \ \ \ \ \ \ \ \ \ \ \ }0\text{\ \ \ \ \
\ \ \ \ \ \ \ \ \ \ \ \ \ \ \ \ \ \ \ \ \ \ \ \ \ \ \ \ \ \ \ \ \ \ \ \ \ \
\ \ \ \ \ \ \ \ \ \ \ \ \ \ \ \ \ \ \ \ \ \ \ \ } \\ 
\text{For }n=r,\text{ \ \ \ \ \ \ \ \ \ \ \ \ \ \ \ \ \ \ \ \ \ \ \ \ \ \ \
\ \ \ \ \ \ \ \ \ \ \ \ \ \ \ \ \ \ \ \ \ \ \ \ \ \ \ \ \ }r!\text{\ \ \ \ \
\ \ \ \ \ \ \ \ \ \ \ \ \ \ \ \ \ \ \ \ \ \ \ \ \ \ \ \ \ \ \ \ \ \ \ \ \ \
\ \ \ \ \ \ \ \ \ \ \ \ \ \ \ \ \ \ \ \ \ \ } \\ 
\text{For }n>r,\text{ \ \ \ \ \ \ \ \ \ \ }\frac{1}{n-r-1}\left[ \left(
n-1\right) x\frac{\partial ^{r}F_{n-1}(x,y)}{\partial x^{r}}+y\left(
n-1+r\right) \frac{\partial ^{r}F_{n-2}(x,y)}{\partial x^{r}}\right] \text{
\ \ \ }%
\end{array}%
\right. .
\end{equation*}
\end{corollary}

\bigskip 

\section{Acknowledgements}

The authors thank to the referee for a through reading and benefical
suggestions.

\end{document}